\documentclass[11pt]{amsart}
\usepackage{amsmath, amsthm, amssymb}    
\usepackage{graphicx,color}                    
\usepackage{hyperref}                    

\usepackage[T1]{fontenc}

\newcommand{\tA}{\mathbf{\theta_a}}
\newcommand{\tB}{\mathbf{\theta_b}}
\newcommand{\tC}{\mathbf{\theta_c}}

\newcommand{\bbn}{\mathbf{n}}

\newcommand{\real}{\mathbb{R}}

\newcommand{\dd}{\mathrm{d}}

\newcommand{\beq}{\begin{equation}}
\newcommand{\eeq}{\end{equation}}
\newcommand{\beqas}{\begin{eqnarray*}}
\newcommand{\eeqas}{\end{eqnarray*}}

   \newtheorem{theorem}{Theorem}

   \newtheorem{corollary}[theorem]{Corollary}
   \newtheorem{lemma}[theorem]{Lemma}
   \newtheorem{definition}[theorem]{Definition}

 \theoremstyle{remark}
   
   \newtheorem{remark}[theorem]{Remark}

\newcounter{case}

\begin{document}
\title{Tileable Surfaces}
\author{David Brander}
\address{Department of Applied Mathematics and Computer Science,\\
Technical University of Denmark\\
 Matematiktorvet, Building 303 B\\
DK-2800 Kgs. Lyngby\\ Denmark}
\email{dbra@dtu.dk}

\author{Jens Gravesen}
\address{Department of Applied Mathematics and Computer Science,\\
Technical University of Denmark\\
 Matematiktorvet, Building 303 B\\
DK-2800 Kgs. Lyngby\\ Denmark}
\email{jgra@dtu.dk}
\date{\today}				

\thanks{Research partly supported by  Horizon-EIC-2023 pathfinder challenge 01 project number 101161085 \emph{STACK}}
\begin{abstract}
We study $C^1$-regular surfaces in $\real^3$ that admit tilings by a finite number of rigid motion congruence classes of tiles.   
We construct examples with various topologies and present a framework for a systematic study, mainly concentrating on
monotilings.
A \emph{finite edge} prototile is a tile that has only a finite number of possible interfaces with adjacent copies of
itself.  We describe all monotilings by such tiles with three or less edges.
We consider the question of whether a monohedral polyhedron can be ``smoothed'' to become a finite edge type
tileable surface with the same graph structure, and we give an example where
this is not possible.   Finally we list some open problems.

\end{abstract}

\maketitle

\section{Introduction}
Tilings of the plane, sphere and hyperbolic space are well known.
But if we consider physical manifestations of tiling, that is, a figure or object composed of many copies of the same shape, or a limited number of distinct shapes, there is no special reason why a given tile should be of constant curvature, or that the total space need be homogeneous {(i.e., possessing global isometries mapping any point to any other point)}. In manufacturing and architectural scenarios, there is  an economic case for constructing a surface from many copies of a limited number of distinct shapes or tiles.  Recently, there has been some serious work related to this problem, in the special case of polyhedral surfaces and variants -- see the literature review in Section \ref{review} below.

\begin{figure}[h!tbp]
	\begin{center}
				\includegraphics[height=26mm]{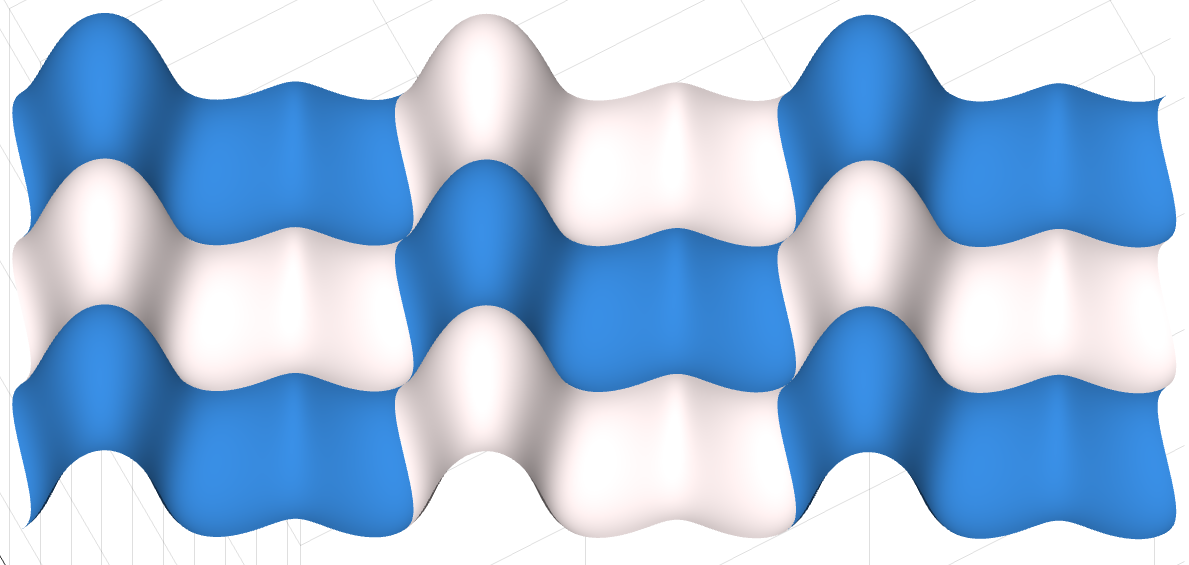}  	 \quad    \quad
  \includegraphics[height=26mm]{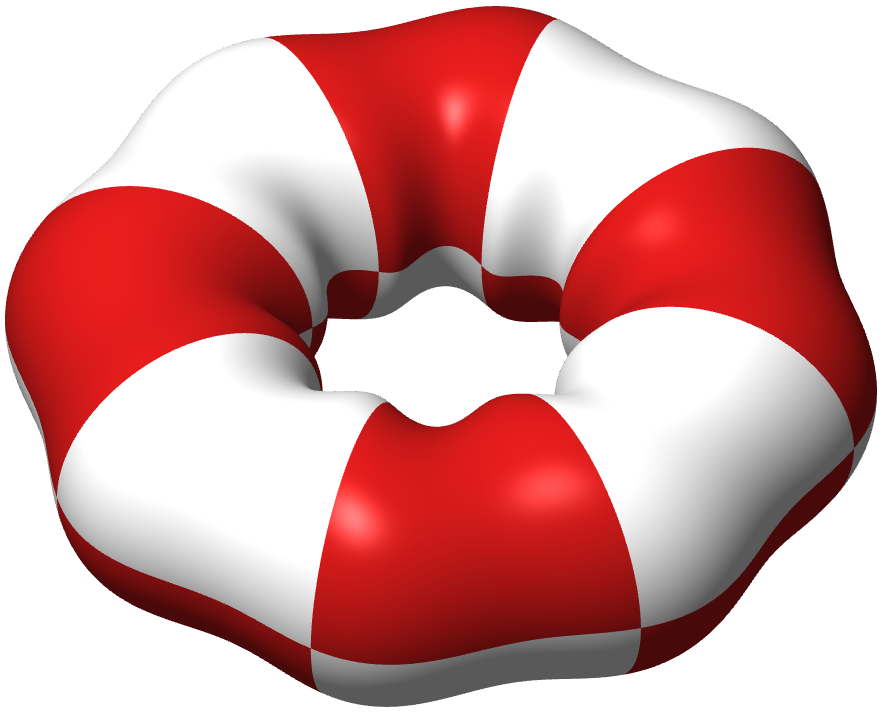}  				
   \quad   \quad
  \includegraphics[height=26mm]{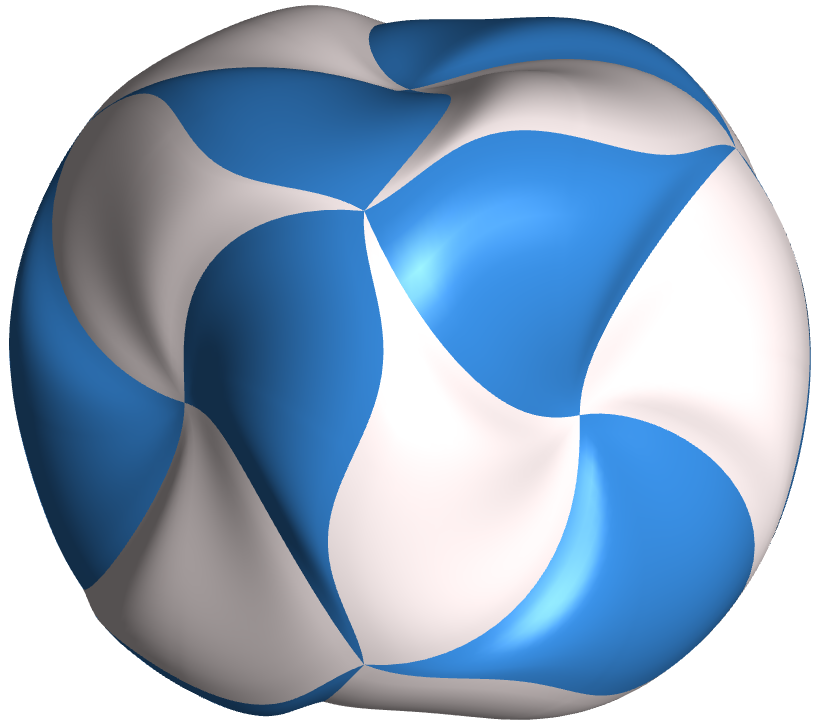}  														
		\end{center}
	\caption{Examples of monotiled surfaces.}
	\label{fig1}
\end{figure}

In this article we study a generalization of tilings to smooth, non-homogeneous, surfaces.
Unlike the case of polyhedra, we require the tiles to fit together in a tangent
continuous way to  form a regular curved surface of class $C^k$, where, $k \geq 1$, such as those shown in Figure \ref{fig1}.
Our goal is to set up a framework for studying these surfaces and to obtain a sense for how they are related to
tilings of the plane or sphere (which they obviously generalize) and to monohedral polyhedral surfaces.

\begin{definition} \label{tilingdef}
Let $S$ be a closed, connected,  $C^k$-embedded surface, possibly with boundary, in $\real^3$.
We will say that $S$ is \emph{(mono)tileable} by the \emph{prototile} $T$, if it can be written
\beq \label{decomp1}
S=\cup_{i=1}^{m} T_{i},
\eeq
where {$m \in {\mathbb N} \cup \{ \infty\}$ is an extended natural number} and,  
\begin{enumerate}
\item \label{cond1} the prototile $T \subset \real^3$ is homeomorphic to a closed topological disc,
\item \label{cond2} distinct sets $T_{i}$ (tiles) intersect only in subsets of their boundaries,
and any tile intersects only a finite number of other tiles,
\item \label{cond3} each of the sets $T_{i}$, is congruent to $T$ via a rigid motion of the ambient space $\real^3$.
\end{enumerate}
If $S$ is oriented and all the rigid motions $T \to T_i$ are orientation preserving, then we call the decomposition an \emph{oriented tiling}, and $S$ an \emph{orientably tileable} surface.
\end{definition}
Definition \ref{tilingdef} can be generalized in the obvious way to the case that there are $n$ model 
tiles instead of $1$ (an \emph{$n$-tileable} surface), but we will mainly discuss the case $n=1$ in this article.
We will also focus on \emph{oriented} tilings, because these correspond to actual physical tiles which have two distinct sides separated by some thickness, so that one side is at least slightly different from the other geometrically, and may also differ even further for practical reasons.   

Even if we allow for a large (finite) number $n$ of prototiles, it is easy to see that there are restrictions on the global shape of a tileable surface. For instance {(see, e.g.  \cite{docarmo1, lee2003smooth} for terminology)}:
\begin{theorem}
If $S$ is a \emph{complete} $C^2$-embedded $n$-tileable surface with non-vanishing curvature, 
then $S$ is a topological sphere.
\end{theorem}
To see that this holds, note that,  since $S$ is tiled by a finite number of distinct compact prototiles, the curvature must achieve both its maximum and minimum values, and hence be bounded away from zero.  By Efimov's theorem \cite{efimov1968hyperbolic}, there is no complete $C^2$-immersed surface in $\real^3$ with curvature bounded above by a negative constant.
If the curvature is positive  the surface must be compact by the Bonnet-Myers theorem, and hence a topological sphere by the Gauss-Bonnet theorem.  

Similarly, given a $C^2$ monotiled surface, although it is obviously possible to deform the prototile and apply the corresponding deformation to all the other copies to remain a $C^2$ monotiled surface with the same topology and face combinatorics, there are nevertheless restrictions on the deformed shape. By the Gauss-Bonnet theorem we have:
\begin{theorem}
If $S$ is a compact $C^2$ monotiled surface with Euler characteristic $\chi(S)$, then the integral of the Gauss curvature
over the prototile is given by:
\[
\int_T K \dd A = \frac{2\pi}{m} \chi(S),
\]
where $m$ is the total number of tiles.
\end{theorem}

Below we will define an \emph{admissible monotile} to be a tile that can be completely surrounded by copies of
itself, and we will only be interested in the case that the prototile is admissible, as is automatically the case for a tileable surface without boundary. 

There are apparently three distinct cases of interest:
\begin{enumerate}
\item Tileable surfaces with boundary.
\item Compact tileable surfaces without boundary.
\item Complete non-compact tileable surfaces.
\end{enumerate}
In general, an admissible tile generates the first type, just by surrounding the tile with copies of itself, except for
special cases such as a hemisphere, two copies of which fit together to form a surface without boundary.  In manufacturing and architecture surfaces with boundary are the norm, so
tileable  surfaces with boundary are relevant even when the tiling does not extend to a complete surface.

 \subsection{Relation between tileable surfaces and tilings of the plane or sphere}
The many known monotilings of the plane and sphere are examples of monotileable surfaces.
 Deforming the prototile in such a way that the tiles still fit together in a tangent continuous 
 way can easily be performed (Figure \ref{fig2}, third image) to give many examples of monotileable surfaces. However,
 if done in a straightforward way as here, the surface will continue to have the global appearance of a plane or sphere: for the planar case there will still be a single plane that intersects every tile, and likewise a sphere that intersects every tile in the spherical case.
 \begin{figure}[h!tbp]
	\begin{center}
	\includegraphics[height=25mm]{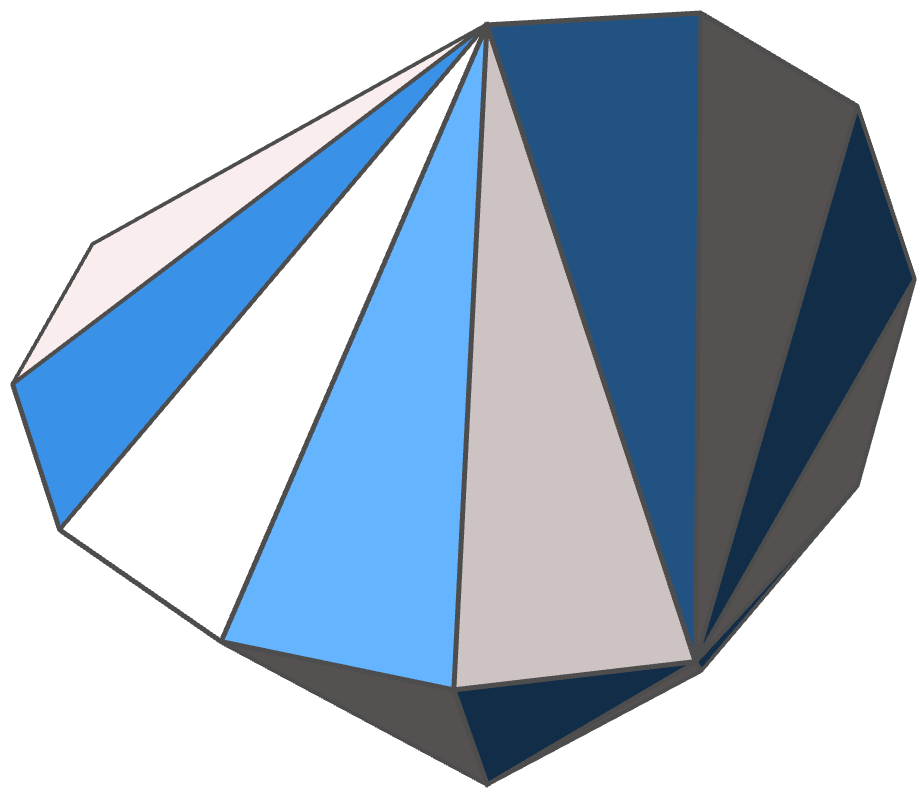}    \quad  
				\includegraphics[height=27mm]{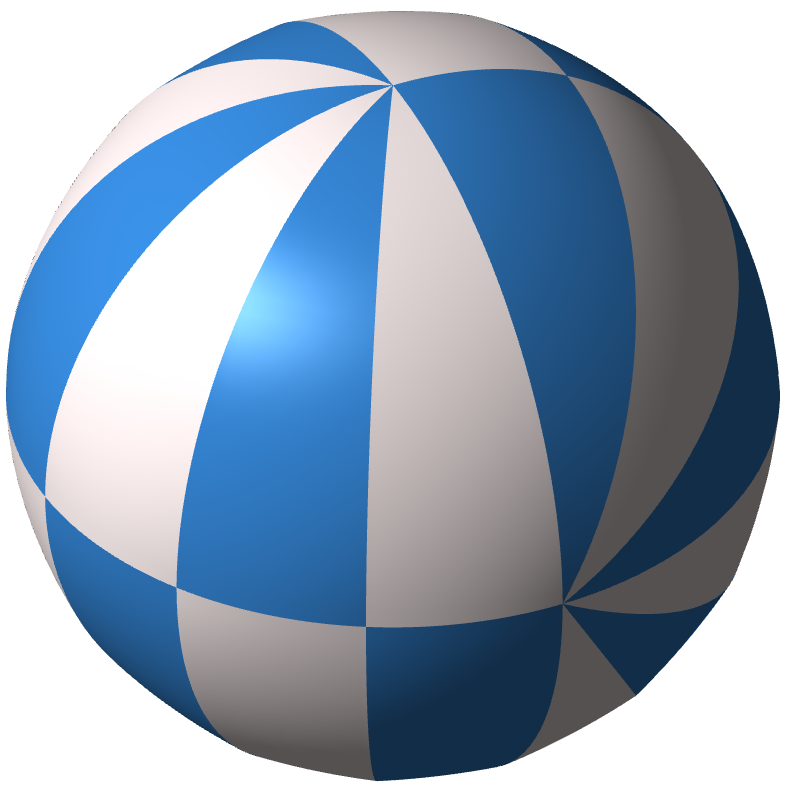}  	 \quad  
  \includegraphics[height=27mm]{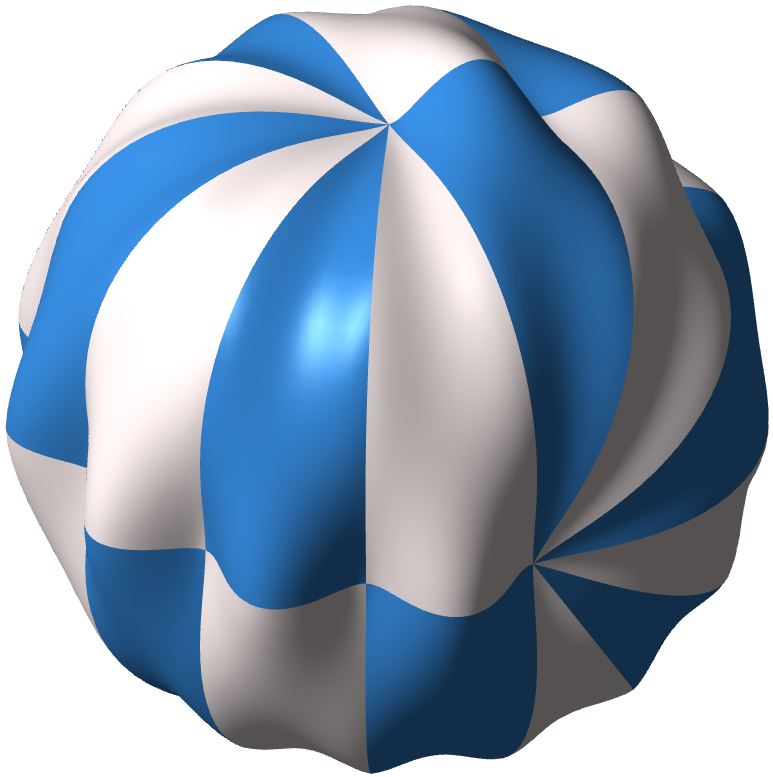}  	\quad
    \includegraphics[height=27mm]{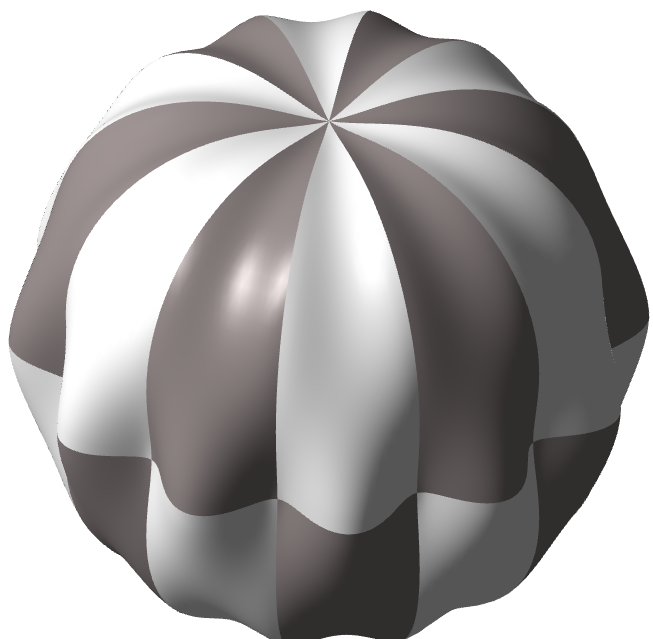}  												
		\end{center}
	\caption{From left to right: monohedral polyhedron (biarc-hull \cite{eppstein2021polyhedral});  projection to the sphere; more general $C^1$ tiling; geometrically distinct tiled surface with the same prototile.}
	\label{fig2}
\end{figure}

\subsection{Relation between tileable surfaces and monohedral polyhedra}
The term \emph{monohedral polyhedron} refers to a polyhedron with congruent faces.  Improper rigid motions are usually allowed.

 \begin{figure}[h!tbp]
	\begin{center}
	\includegraphics[height=30mm]{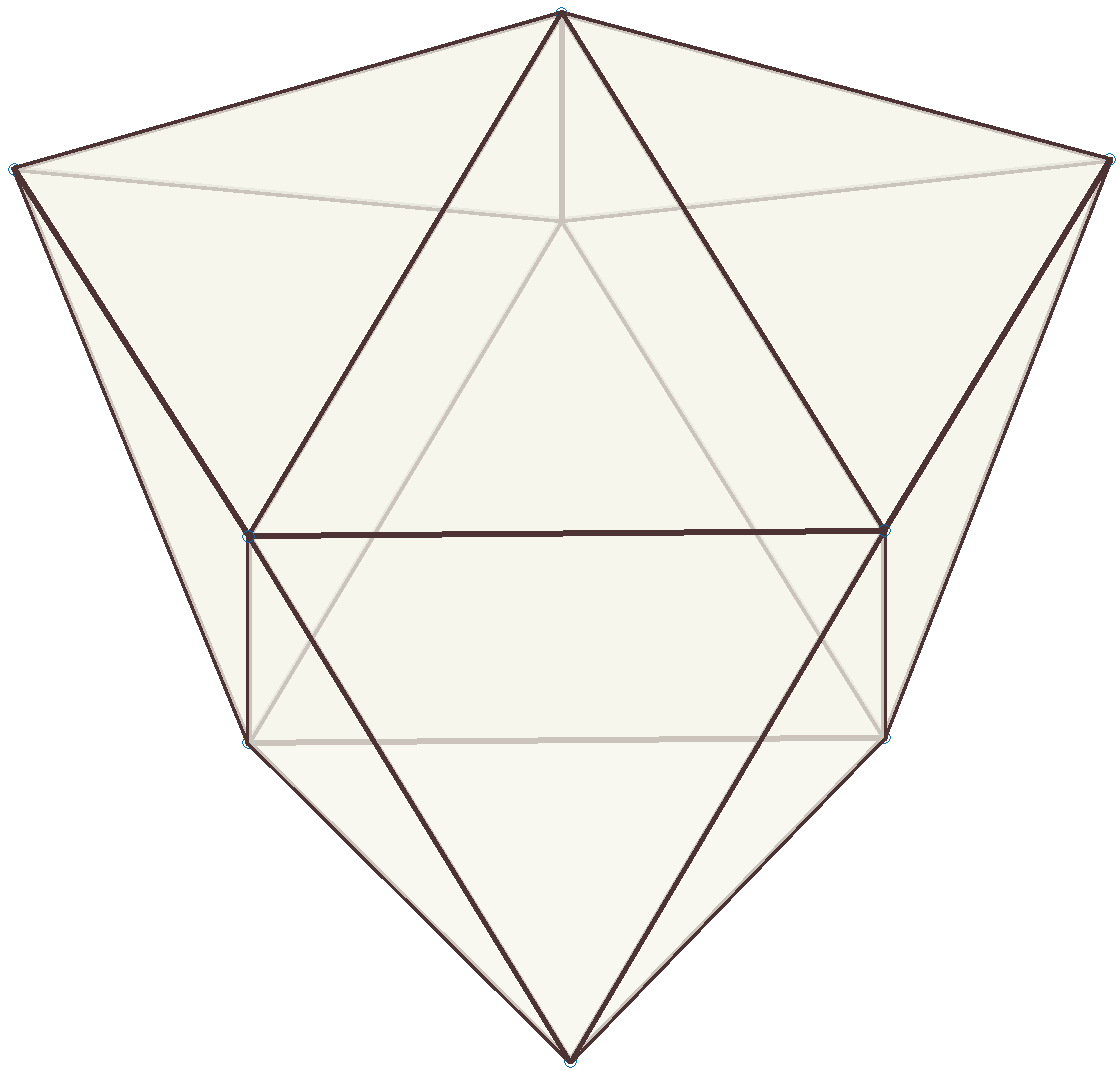}   				
		\end{center}
	\caption{The triaugmented triangular prism 
	 \textbf{cannot} be smoothed to a finite edge oriented monotiled surface with the same graph structure}
	\label{fig3}
\end{figure}

One idea for producing examples of tileable surfaces is to take a monohedral polyhedron and, if possible, deform the model face until the faces fit together tangent continuously along the edges (Figure \ref{fig2}).  {This is obviously possible in some cases.  However, we
will show that there are also examples that \emph{cannot} be smoothed in this way (See Theorem \ref{ttp} below and Figure \ref{fig3}).
Thus, a basic problem is to understand which monohedral polyhedra can be smoothed to become regular tiled surfaces.}

 {Conversely, given a tiled surface,  if we define an ``edge'' on the boundary of a tile to be an uninterrupted interface between two tiles,  and call its endpoints ``vertices'',  then the vertices may correspond to those of a  polyhedral surface. This can be done for the first surface in Figure \ref{fig1}, to form a tiled plane.   For the torus shown in the same figure, the resulting polyhedron collapses because the vertices are
all in the same plane;  more generally, if the number of vertices  on a face is more than three, they will not usually be co-planar, 
and moreover for some tiled surfaces (Figure \ref{fig:nonrigid}) 
there are no canonical points to choose as vertices on the tile.}


\section{Related prior work}  \label{review}
We give here a brief overview of what, for various reasons, appear to be the most relevant previously studied concepts.

\subsection{Doubly and triply periodic surfaces}
These are true tileable surfaces, according to our definition.
Any 2-dimensional or 3-dimensional lattice structure gives us many examples of monotiled surfaces, simply by placing
a surface in the fundamental domain in such a way that translated copies of it are joined tangent continuously.
A deformation of  a periodic plane-tiling (Figure \ref{fig1},  left) is an example of a doubly periodic surface.   More generally, so-called \emph{triply periodic surfaces} in $\real^3$ can easily be constructed by placing
a suitable surface inside  the fundamental domain of a 3-dimensional lattice.
It is possible to arrange that the fundamental piece is also tiled by congruent copies of a topological disc, 
thus the entire surface satisfies our definition of a monotiled surface.  Triply periodic \emph{minimal} surfaces, first discovered by Schwarz and his student in the 19th century, 
are well known (see, e.g. \cite{meeks1990theory}, \cite{karcher1996construction}, \cite{grosse2012triply}) and have  recently become very popular in practical applications such as in additive manufacturing and in tissue engineering:  see, e.g., \cite{rajagopalan2006schwarz}, \cite{feng2022triply} and many other recent works.

Apart from the shape of the fundamental piece, the overall geometry of all of these examples is very simple however, as it is determined by a lattice structure.

\subsection{Tilings of constant curvature surfaces}
This topic is well known, and extensively studied.  See \cite{grunbaum1987tilings}, \cite{schulte1993tilings}, \cite{toth2017handbook} and \cite{adams2023tiling}, for an introduction to tilings,  and \cite{zong2020can} for a survey on monotiling of the plane.
Common themes are classifications of patterns, group actions, tilings with as small a number of distinct tiles as possible,  non-repeating tilings, tilings with different edge colorings etc. However,  we have been unable to find any systematic attempt to extend this study to smooth surfaces of non-constant curvature.

\subsection{Polyhedral surfaces}
{It follows from a theorem of Steinitz that every 3-connected planar graph can be realized as a convex polyhedron.}
However, there are restrictions on the geometry: for instance, not every such graph can be realized as a \emph{monohedral} polyhedron.  Gr\"unbaum \cite{grunbaum2001convex} shows that there exist triangulated spheres that are not realized by any acoptic (meaning homeomorphic to a sphere) monohedral polyhedron.

A monohedral polyhedron is called an \emph{isohedron} if all faces lie within the polyhedron's symmetry orbit, i.e., the rigid motions of the entire polyhedron act transitively on the faces.  Examples include the Platonic solids, the Catalan solids, the bipyramids and the trapezohedra. 
Convex isohedra  are classified, with a list of 30 types.
Even in the convex case,  general monohedral polyhedra do not appear to be classified.
Examples that are not isohedra include the triaugmented triangular prism (Figure \ref{fig3}),  which has equilateral triangular faces,  and the rhombic icosahedron
which has 20 rhombic faces.
 Eppstein \cite{eppstein2021polyhedral} gives an infinite family of monohedral polyhedra monotiled by isosceles triangles.

\subsection{Polyhedral and semidiscrete approximations} \label{sec:approximations}
 In computer aided design and architecture, it is commonly desirable to \emph{approximate} a target (smooth) surface by a polyhedral mesh, or some kind of
semi-discrete structure,  with as few as possible distinct elements, to reduce the cost of manufacturing.

Eigensatz et al \cite{eigensatz2010paneling} describe the ``paneling problem'' where a free form surface is to be approximated by as few as possible distinct panels to minimize the number of molds to be used in the construction.
Liu et al \cite{liu2023reducing} study a similar problem, using polygonal faces.
Singh  and Schaeffer \cite{singh2010triangle} give an algorithm where the input is an arbitrary triangulated surface and the output is an approximation using a small number of polygonal faces. Fu et al \cite{fu2010k} give an algorithm for optimization of a quadmesh approximation of a surface using a small number of distinct (non-planar) quadrilaterals.

Chen et al \cite{chen2023masonry} consider a similar problem for a \emph{shell structure}: approximate the surface by a piecewise linear shell structure, where the goal is to have many identical parts. The difference from polyhedral surfaces is the thickness.  

Bo et al \cite{bo2011circular} (see also \cite{bartovn2013circular}), replace polyhedral surfaces with \emph{circular arc} structures, that is surfaces defined by 2D mesh combinatorics,  where the edges are circular arcs that meet with a common tangent plane.     
Their study shares features of what we discuss below, in that the edges of the faces are tangent continuous. However they do not consider surfaces with congruent faces - instead they aim to construct surfaces with congruent \emph{vertices}.

\subsection{Mosaic methods}
Another approximation concept for smooth surfaces,  mosaic methods  \cite{PW2008}, \cite{PW2009} \cite{chen2017fabricable},\cite{hu2015surface}, 
means approximating a 3D surface by covering as much as possible of it with
flat tiles drawn from a given collection.  Small gaps are allowed between the tiles.  The measure of success for a tile set together with an algorithm is to minimize the gap area while closely approximating the surface.

\subsection{Tiling inspired decoration of arbitrary surfaces}
Any surface has the conformal structure of one of the standard 3 constant curvature
space forms, which allows for texture maps \emph{conformally} equivalent to the many known tilings
of space forms (see, e.g. Gu et al \cite{gu2010fundamentals}). These are not tilings in the sense
discussed here, however, since conformal maps preserve angles but not size.

Jiang et al \cite{jiang2015polyhedral} explore the idea of extending the common tiling patterns on the plane to polyhedral surfaces by another idea.  The surface 
tiles are defined by lifting planar tiles onto paraboloids (approximating the surface around a point by the curvature paraboloid). The method gives a natural way to produce tiling-like patterns on general surfaces, but the tiles are not congruent to each other. The shape of the lift to the surface of a particular model tile will depend on the curvature of the surface at that point.


\pagebreak
 
\section{Rigid tilings} \label{rigidity}
We now investigate some of the properties of tileable surfaces defined in Definition \ref{tilingdef}.
\begin{definition} Two tiles on a tiled surface are \emph{neighbours} if they intersect.   
It follows from Definition \ref{tilingdef} that they must intersect at a finite number of points or closed intervals.
If the intersection contains an interval larger than a single point, the tiles 
are called \emph{adjacent}.
\end{definition}
Suppose that $S$ is a  monotileable surface without boundary.  
 Since the surface has no boundary, every copy of the prototile $T$ (which is homeomorphic to a closed disc) must be surrounded by other copies of the tile.   More generally, for a tileable surface with boundary, a tile $T_j$ that contains no boundary points of the surface will be called an \emph{interior tile}.  Such a tile $T$ will be called rigid if there is only one way  to surround it by copies of itself so that the resulting surface is $C^1$-regular along the boundary of $T$.  More precisely:
 \begin{definition}
 Let $T$ be a continuously embedded closed topological disc in $\real^3$,  $C^k$ regular on the interior of $T$.
 We call $T$ an \emph{admissible tile} if there exists a $C^k$ tileable surface $S$ that contains $T$ as an
 interior tile.  
 An admissible tile  $T$ is \emph{rigid} if:
 \begin{enumerate}
 \item There is a unique set $\{ T_1, T_2, \dots , T_N\}$ that constitutes the set of neighbours of $T$ in \emph{any}
 tiled surface containing $T$ {as an interior tile}.
 \item There is a \emph{unique} rigid motion $\phi_i: T  \to T_i$ for every $i = 1, \dots N$.
\end{enumerate}
\end{definition}
 \begin{figure}[h!tbp]
	\begin{center}
	  \includegraphics[height=20mm]{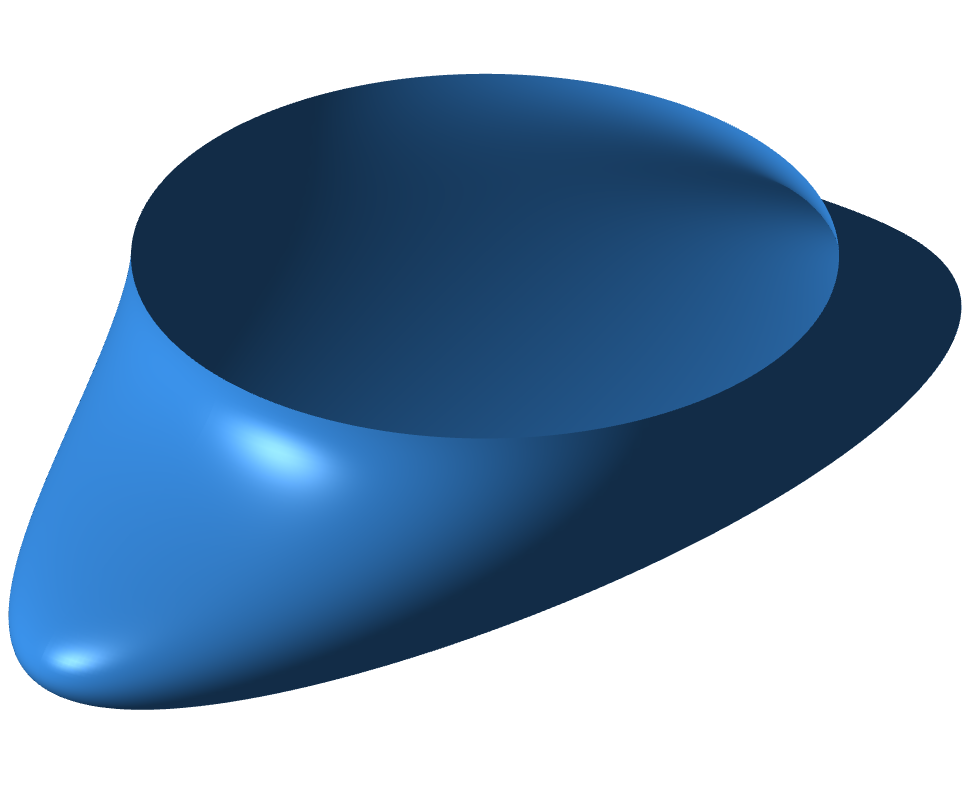}  \quad \quad
  \includegraphics[height=24mm]{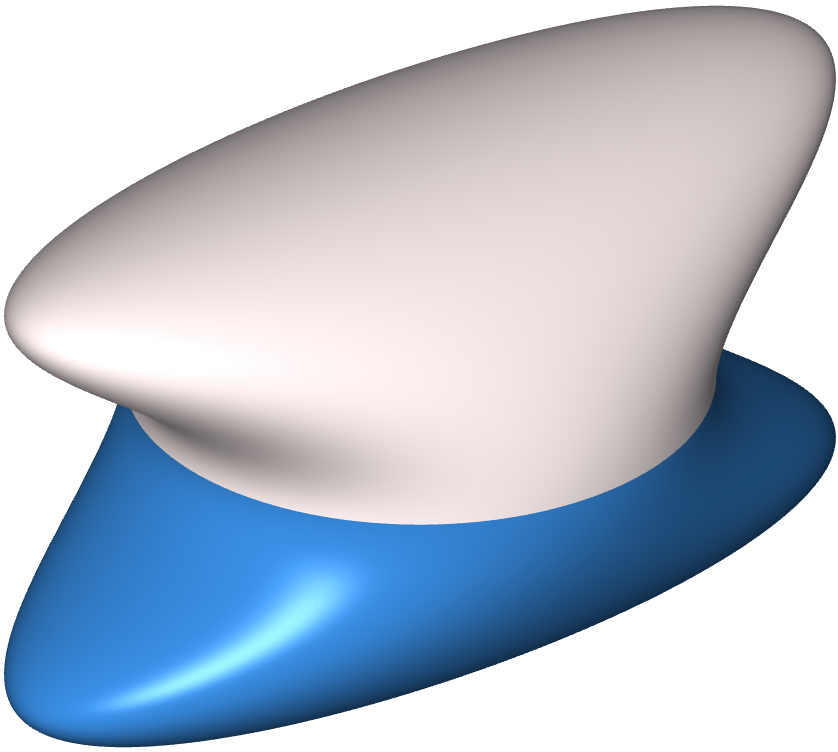}  	\quad \quad 
    \includegraphics[height=24mm]{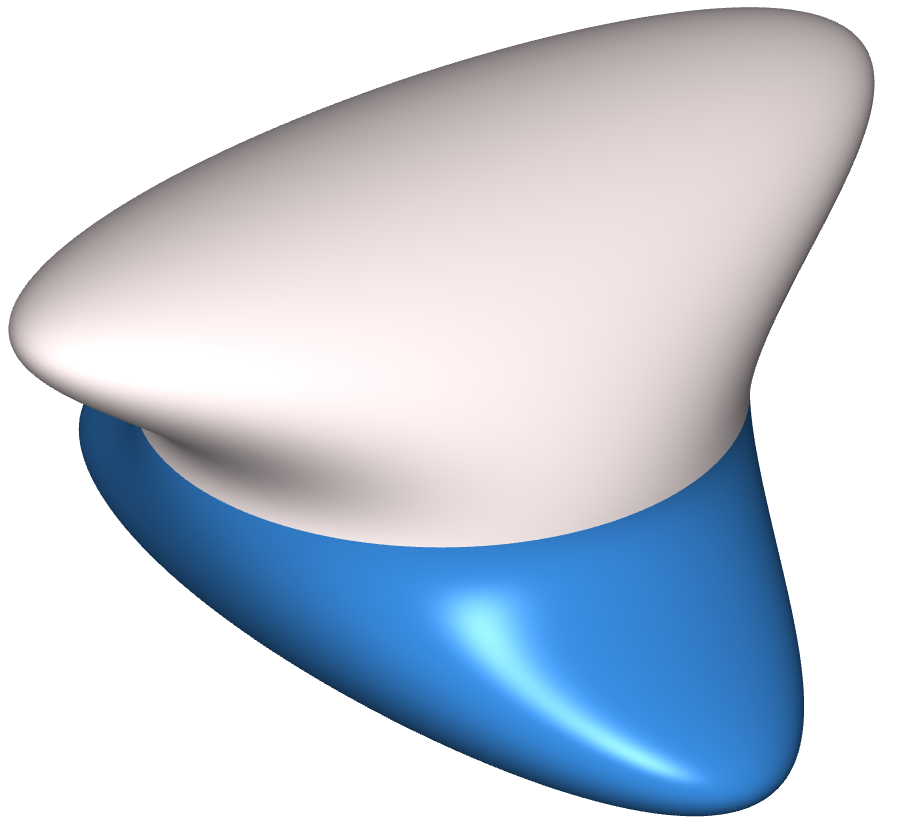}  												
		\end{center}
\caption{A tile that is not rigid and two different surfaces tiled by it.} \label{fig:nonrigid}
\end{figure} 

The tiles in Figure \ref{fig1} are rigid, whilst
Figures \ref{fig2} and \ref{fig:nonrigid} show examples of tiles that are not rigid.  
The tile in Figure \ref{fig2} can be surrounded by itself in only a finite number of ways.  The tile in Figure \ref{fig:nonrigid} has a circle as its boundary, and {surface normal on the boundary parallel to the radius of the circle}. There are uncountably many ways to fit two copies together.   

The next result shows that, for a rigid prototile, there is a transitive group action by isometries on the tiles.
However,  the converse does not hold, as illustrated by the last surface shown in Figure \ref{fig2}, which is face transitive, even though the prototile is not rigid.
 \begin{theorem} \label{rigiditythm}
 Let $S$ be a monotiled surface without boundary such that the prototile $T=T_1$ is rigid.
 Then 
 \begin{enumerate}
\item Every tile in the tiling is also rigid.
 \item \label{item1} If $\psi: T_j \to T_k$ and $\phi: T_j \to T_k$ are two rigid motions that take a tile $T_j$ to
 a tile $T_k$ then $\psi = \phi$.
 \item \label{item2} Every isometry $\psi: T_j \to T_k$ among the tiles is a rigid motion of the entire surface $S$, and
 preserves the tile structure.
 \item The set of such isometries forms a subgroup $G$, of the group $SE(3)$ of rigid motions, that acts transitively
 on the tiles of $S$, and is generated by the rigid motions $\phi_{i_j}: T \to W_{i_j}$, where $W_{i_j}$ are the tiles \emph{adjacent} to $T$.
 \end{enumerate}
 \end{theorem}
 \begin{proof}
 \begin{enumerate}
 \item Every tile $T_i$ is interior and is equivalent to $T$ via a rigid motion $\psi_i: T \to T_i$.
Applying $\psi_i$ to the neighbours $W_1, \dots W_N$ of $T$, we obtain that $T_i$ is also  rigid, with neighbours, $\psi_i(W_1),\dots \psi_i(W_N)$.
 \item 
 If $\psi: T_j \to T_k$ and $\phi: T_j \to T_k$  are not equal, then $\xi := \psi^{-1} \circ \phi$ is a non-trivial rigid motion that takes $T_j$ to itself. But then the maps $\phi_i: T_i \to W_i$ that map $T_i$ to its neighbours are not unique: 
 they can be replaced by $\phi_i \circ \xi$, contradicting the rigidity of $T_i$.
 \item
   For any tile $T_i$, there is a rigid motion $\psi_{i}: T=T_1 \to T_i$.    The set
   $\psi_i(S)$ is itself a tiled surface isometric to $S$. We need to show that $\psi_i(S)=S$ and that the tiles are permuted by $\psi_i$.  If $W_1, \dots W_N$ are the neighbours of $T$, the rigidity of $T_i$ implies that $\psi_i\left(T \cup \bigcup_{j=1}^N W_j \right)$ is a subset of $S$ \emph{with the
 same tile structure}, since
 both are tileable surfaces that surround $T_i$.
 We can apply the same argument to the neighbours of $W_j$ and beyond, to show that, for 
 any chain of neighbouring tiles $T_{i_1}, \dots T_{i_k}$,  starting from $T$, we have 
 \[
 \psi_i\left(\bigcup_{j=1}^k T_{i_j} \right) \subset S,
 \]
 and that each set $\psi_i(T_{i_j})$ is in fact a tile of $S$. The claim \ref{item2} then follows
 from the connectedness of $S$ {(see Definition \ref{tilingdef})}.  Given any point $x$ in $S$, there is a path from $x$ to some point
 in $T$. This path traverses a finite number of neighbouring tiles, including any tile containing $x$. 
 Hence any tile containing $x$ can be taken to be one of the $T_{i_j}$ above, and 
 is mapped by $\psi_i$ to some other tile of $S$.
 \item We showed that every tile transformation $\psi_{ij} \in SE(3)$ that takes $T_i$ to $T_j$ is also a tile
 transformation among the rest of the tiles.  Thus, if $\psi_{ij}$ and $\psi_{kl}$ are transformations between arbitrary pairs of tiles, we can compose them because $\psi_{ij}$ takes the tile $T_l$ to some tile $T_m$, so
 $\psi_{ij} \circ \psi_{kl}$ is a rigid motion that takes $T_k$ to $T_m$.  It follows that the  tile transformations form a subgroup $G$ of $SE(3)$.
 The fact that $G$ is generated by the isometries $\phi_{i_j}$ to \emph{adjacent} tiles is clear, since any tile $T_k$ can be 
 reached by a chain of adjacent tiles, and the tiling map $T \to T_k$ is unique. 
\end{enumerate}
 \end{proof}

\section{Finite edge type oriented tiles}
From now on we assume that all tilings are oriented monotilings.

In order to analyze how tiles can fit together, we want to decompose the boundary of the prototile into subsets where pairs of adjacent tiles meet, analogous to edges on a polyhedral surface.

For a given tiled surface, the notion of a vertex usually refers to a point that is contained in three or more tiles. 
When considering a tile on its own, unless the tile is rigid, this notion is not well defined.
Hence, we give a more general definition for edges and vertices:
\begin{definition}
Let $T$ be an admissible tile. A \emph{corona} for $T$ is a monotileable surface $S$ made of congruent copies $T_i$ of $T$, such that $T\cap T_i$ is non-empty for all $i$, and such that $T$ is an interior tile of $S$.
An admissible tile is of \emph{finite (edge)} type if there is only a finite number of ways to form a corona for $T$.
  Any closed curve segment along the boundary of $T$, other than a point, that appears as 
a connected component of the intersection of two tiles in some corona, is called an \emph{edge}, and its endpoints are called \emph{vertices}.
\end{definition}
Since there are only a finite number of coronas for a finite type tile (and any corona is formed by 
a finite number of tiles), there are, altogether, only a finite number of edges.   

Note that the property of being of finite edge type depends not only on the boundary of the tile, but also on the surface normal along the boundary. Figure \ref{fig:nonrigid} shows a tile that is \emph{not} finite type.  A planar triangle is also \emph{not} of finite edge type, as there are infinitely many ways to surround it by sliding along the edges.

Give the boundary of $T$ an orientation, and assume that each edge $e_j$ is parameterized 
as a continuous map $e_j: [0,1] \to T$,  in accordance with this orientation.   Every edge of 
an admissible tile must match an edge of an adjacent tile, with the reverse orientation (Figure \ref{config1}, left).
Because each tile is obtained by a rigid motion from any other,
it follows that for edge $e$, the \emph{reverse} curve $\bar e$, obtained by parameterizing $e$ 
with the opposite orientation, must, up to a rigid motion also be included among the edges. Thus,  the set of geometrically distinct edges for an oriented tile must be a list (not necessarily deployed in this order) of the form:
\[
a_1, a_2, \dots a_m,  \quad \tilde a_1, \tilde a_2, \dots \tilde a_m,   \quad \alpha_1, \alpha_2, \dots \alpha_n,
\]
where, if we write $e_1\equiv e_2$ to denote that there is a rigid motion mapping the edge $e_1$ to the 
 edge $e_2$, preserving the surface normal along the edge, we have:
\[
\bar a_j \equiv \tilde a_j, \quad \quad \bar \alpha_k \equiv \alpha_k, \quad \hbox{ for all } j, k.
\]
\textbf{Notation:} We will describe the \emph{(edge) type} of the prototile by listing the edges anticlockwise
in order, using $\bar e$ for an edge that is congruent to the reverse of edge $e$, and using Greek letters for edges that 
are self congruent (Figure \ref{config1}, left and middle).   When tiles are placed together, we will describe a vertex 
type by listing the edges coming out of the vertex, anticlockwise in order, enclosed in parentheses (Figure \ref{config1}, right).
\emph{Note: The surface normals along the edges are also preserved by the congruences relating identically  labelled edges}. 

\begin{figure}[h!tbp]
	\begin{center}
	\includegraphics[height=20mm]{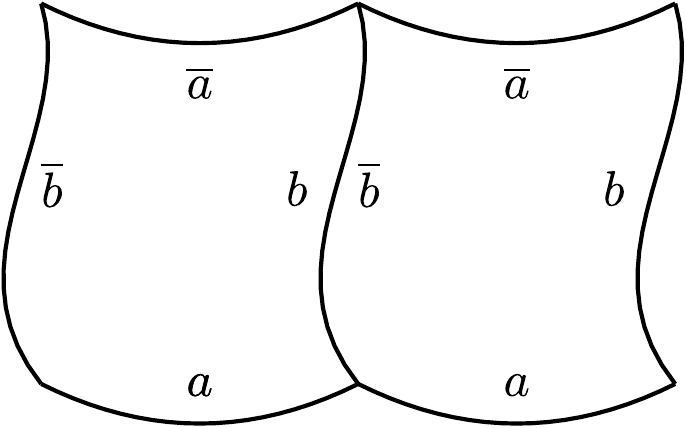}    \quad   \quad  
	\includegraphics[height=20mm]{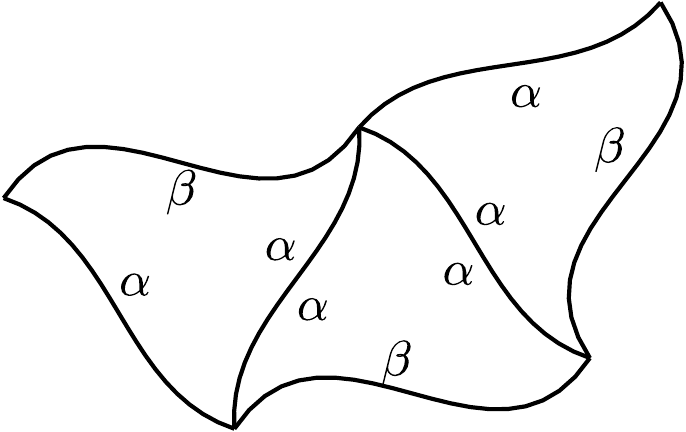}    \quad  \quad  
		\includegraphics[height=20mm]{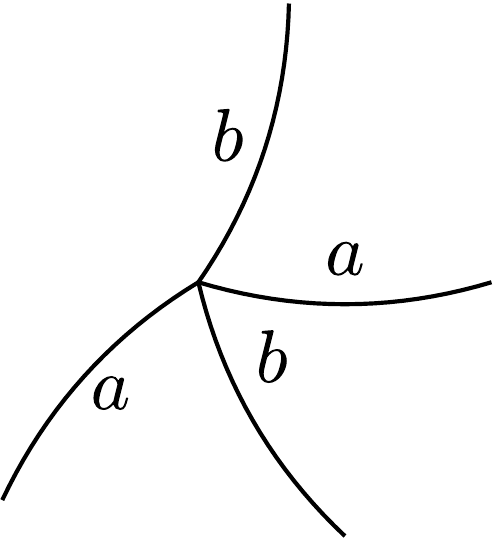} 
		\end{center}
	\caption{Left to right: edge type $ab\bar a\bar b$, edge type $\alpha \alpha \beta$, vertex type $(abab)$.}
	\label{config1}
\end{figure}

Note that a self congruent edge $\alpha$ is symmetric under a unique rigid motion
$\phi$ that interchanges the end points and satisfies $\phi^2=I$.  Necessarily, $\phi$ is a 
rotation of angle $\pi$ about an axis through the midpoint of the line segment joining the endpoints  of $\alpha$.
\begin{remark}
If one wishes to study \emph{non}-oriented tileable surfaces, then another type of edge can be added, which is congruent to another edge after a surface orientation reversing rotation about an axis through the midpoint of the edge.  We exclude this possibility in this article.
\end{remark}

\subsection{Smoothability of monohedral polyhedra} \label{smoothability}
Given a finite edge type monotiled surface $S$, the \emph{graph} of $S$ is the graph naturally defined by the tiling: a vertex of the graph is any point on the surface that corresponds to a vertex on any of the tiles that contain it. An edge of the graph is a set of points on the surface that is an edge of a tile that contains it.  Note that the graph of tiling can have vertices with edge-valence two, unlike the case of standard edge-to-edge tilings, where vertices have valence at least 3.  
\begin{definition}
A monohedral polyhedron $X$ is \emph{smoothable} if there exists a finite edge type monotileable surface with the same graph structure for the faces and edges as those of $X$.  
\end{definition}
If a smoothing $\Sigma$ can be found with identical spatial vertex positions as those of $X$, then we call $\Sigma$ a \emph{vertex preserving} smoothing of $X$. 
Any convex monohedral polyhedron the vertices of which lie on a sphere  can be projected onto the sphere to get 
a vertex preserving smooth monotiling of the sphere, as in Figure \ref{fig2}.  

To find a convex monohedral polyhedron that cannot be smoothed, we will first analyze 
possible vertex configurations of a surface tileable by a 3-edge tile.

\begin{figure}[h!tbp]
	\begin{center}
	\includegraphics[height=22mm]{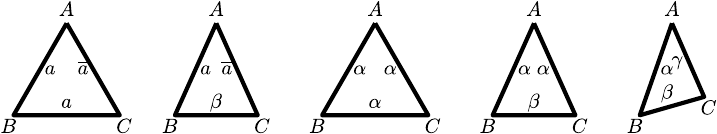}  
		\end{center}
	\caption{Possible edge configurations for a triangular tile.}
	\label{fig:3edges}
\end{figure}

\begin{lemma} \label{vertextypes}
The potential tile types for a $3$-edge oriented prototile are shown in Figure \ref{fig:3edges}.   We roughly
categorize here the possible vertex configurations for oriented tilings with each type:
\begin{enumerate}
\item The case $aa\bar a$ cannot occur as an interior tile
\item Tile type $a \bar a \beta$: two types of vertex are possible one with anticlockwise edge sequence $(a a  \dots)$,  and the other, which must have even degree, with edge sequence $(\beta a \beta a \dots)$. 
\item Tile type $\alpha \alpha \alpha$: all vertices have the same degree.
\item Tile type $\alpha \alpha \beta$.  Vertices are of type
$(\alpha \alpha \dots)$, $(\alpha \beta \dots)$ etc,  where $\beta$ cannot appear twice consecutively.
The tile is in general not rigid, and the same tile can have more than one type of vertex.
\item Tile type $\alpha \beta \gamma$: every vertex has degree a multiple of three, and 
there is only one type of vertex for a given tile.
\end{enumerate}
\end{lemma}
\begin{proof}
Most of the claims can easily be checked by the reader.   The case $\alpha \alpha \alpha$ is more subtle than the others.
In principle,  the three corners $A$, $B$ and $C$ are geometrically distinct, depending on how the face is filled into the edges.
However,  the interior angles at all three corners are identical:  to see this, consider the six rotations that 
exist defining the symmetries of the $\alpha \alpha \alpha$ tile:  by definition there is a rigid motion $\phi_A$ that takes
 the edge $AB$ to the edge $AC$ preserving the surface normals along these edges.  Since both $A$ and the normal $\bbn_A$ 
 at $A$ are fixed, it is a rotation about  $\bbn_A$.  There is also a rotation $\psi$ of angle $\pi$ about an axis through 
 the midpoint between $A$ and $B$ that interchanges the vertices $A$ and $B$ and fixes the edge $AB$.  
 The composition $\phi = \phi_A \circ \psi$ is a rotation that satisfies $\phi(A) = C$ and $\phi(B) = A$, and also preserves the surface normals at those points.   Now consider the underlying \emph{planar} triangle $\Delta ABC$. This is an equilateral triangle, so there is a rotation $R$ about an axis through its midpoint, perpendicular to its plane, that takes the edge $AB$ to the edge $AC$.  It also preserves the  normals at the vertices, because all of these normals  are symmetric with respect to planes bisecting the vertices pairwise, and hence symmetric with respect to 
 the center of the triangle.  Since  $\phi$ and $R$ are both rotations that take $A$ to $C$, $B$ to $A$ and $\bbn_A$ to $\bbn_C$, and the normal $\bbn_A$ does not lie on the line segment $\underline{AC}$, it follows that $\phi=R$.
 The same argument applies to any of the corners, so the rotation $R$ about the center of $\Delta ABC$ preserves all
 of the edges of the tile, and hence the angles at the corners.  From this it follows that the interior angle at each corner must be $2\pi/d$, where $d$ is the degree at a vertex of the tiling, and the degree is fixed by the geometry of the tile.
\end{proof}

We can use this information to investigate the smoothability of a monohedral polyhedron with triangular faces.
The triaugmented triangular prism (Figure \ref{fig3}) is a convex polyhedron with $14$ equilateral triangles as its faces, constructed by attaching an equilateral square based pyramid to each of the three square faces of a triangular prism.  It has 
vertex configuration
$3 \times 3^4 + 6 \times 3^5$, meaning $3$ vertices of degree $4$, and $6$ of degree $5$.

We first look at the edge type $\alpha \alpha \beta$:
\begin{lemma}  \label{aablemma}
Let $\Sigma$ be a compact finite edge type tiled surface with prototile of boundary configuration $\alpha \alpha \beta$ (Figure \ref{fig:3edges}).  Suppose that $\Sigma$ has vertices of  both degree $4$ and degree $5$.
Denote the
interior angles at the corners $A$,$B$ and $C$ of $T$ by  $\tA$, $\tB$ and $\tC$.  Then the interior angles
must satisfy one of the following pairs of equations:
\beqas
 \hbox{Case 1: } &\quad & 
    \left\{ \begin{array}{rll} 5\tA &=2\pi,  \quad \quad&  (\alpha \alpha \alpha \alpha \alpha) \\
            2\tB+2\tC &=2\pi ,   \quad \quad  &(\alpha \beta \alpha \beta)  \end{array}\right.    \\
   \hbox{Case 2: }  &\quad &
        \left\{ \begin{array}{rll} 5\tA & =2\pi, \quad \quad  &  (\alpha \alpha \alpha \alpha \alpha) \\
                   \tB+\tC+2\tA&=2\pi,  &  (\alpha \beta \alpha \alpha)   \end{array}\right.    \\
 \hbox{Case 3: } &\quad
& \left\{ \begin{array}{rll} 2\tB+2\tC+\tA &=2\pi, \quad \quad  & (\alpha \beta \alpha \beta \alpha)  \\
   2\tA +\tB+\tC &=2\pi, \quad \quad   &  (\alpha \beta \alpha \alpha)  \end{array}\right.\\
 \hbox{Case 4: }   & \quad
& \left\{ \begin{array}{rll}  2\tB+2\tC+\tA &=2\pi,  \quad \quad  & (\alpha \beta \alpha \beta \alpha)\\ 
          4\tA&=2\pi , \quad \quad   &  (\alpha \beta \alpha \alpha)  \end{array}\right. \\
 \hbox{Case 5: }& \quad
& \left\{ \begin{array}{rll} \tB+\tC+3\tA &=2\pi,   \quad \quad  & (\alpha \beta \alpha \alpha \alpha) \\
          2\tB+2\tC &=2\pi, \quad \quad   &  (\alpha \beta \alpha \beta) \hbox{ or } (\alpha \alpha \alpha \alpha) \end{array}\right. \\
\eeqas

Cases $1$ to $4$ have only one type of vertex configuration of degree $5$ and one of degree $4$. Case $5$ has only
the vertex configuration $(\beta \alpha \beta \alpha \alpha)$ of degree $5$ but may have either or both 
vertex configurations
$(\alpha \beta \alpha \beta)$ and $(\alpha \alpha \alpha \alpha)$ of degree $4$
\end{lemma}
\begin{proof}
  By matching corresponding edges of the tile, since a $\beta$ edge can only be followed by an $\alpha$ edge,  a vertex of degree $5$ can only be 
of the form 
\[
(\alpha \alpha \alpha \alpha \alpha), \quad (\beta \alpha \alpha \alpha \alpha), \quad \hbox{or }
(\beta \alpha \beta \alpha \alpha),
\] 
and
a vertex of degree $4$ must be of form 
\[
(\alpha \alpha \alpha \alpha), \quad (\beta \alpha \beta \alpha), \quad 
\hbox{or } (\beta \alpha \alpha \alpha).
\]
For the surface to be regular, the tangent plane must be well defined at a vertex, so the sum of the
angles around a vertex must be $2\pi$.  The above listed vertex configurations thus give, for a vertex of  degree $5$:
\beq \label{vertexsum1}
5\tA = 2\pi, \quad \hbox{or} \quad 2\tB+2\tC+\tA = 2\pi, \quad \hbox{or} \quad
\tB+\tC+3\tA=2\pi,
\eeq 
and for a vertex of degree $4$:
\beq \label{vertexsum2}
 2\tB+2\tC=2\pi, \quad \hbox{or} \quad  \tB+\tC+2\tA=2\pi, \quad \hbox{or} \quad  4\tA =2\pi.
\eeq
At least one of the equations from each of \eqref{vertexsum1} and \eqref{vertexsum2} must be 
satisfied.   Enumerating the possibilities that do not result in inconsistencies gives the five cases in 
the statement of the lemma.
\end{proof}

\begin{theorem} \label{ttp}
The triaugmented triangular prism is not smoothable.  In fact
there is no compact finite edge orientably tileable surface without boundary with vertex configuration  $3 \times 3^4 + 6 \times 3^5$.
\end{theorem}
\begin{proof}
Consider the possible vertex configurations described in Lemma \ref{vertextypes}.
Tile type 1 is not interior.  
Type  $\beta \bar a a$ can only have two types of vertex: $(a a  \dots)$ of arbitrary degree, and $(\beta a \dots)$ of even degree.  Therefore, a vertex of degree $5$ must be of type $(a a a a a)$, and there must be six of these.  This means that the corner $A$ (adjacent to $a$ and $\bar a$) appears $30$ times on the surface, which is impossible as the surface has $14$ tiles only. 
Tile type $\alpha \alpha \alpha$ and $\alpha \beta \gamma$ have vertices all of the same degree, so cannot have vertices of both
degree $4$ and $5$.

Finally, the case $\alpha \alpha \beta$ has the options listed in Lemma \ref{aablemma}.
The surface has 14 geometrically identical faces, so each of the corners, $A$, $B$ and $C$ must appear $14$ times exactly altogether,  and there are $6$ vertices of degree $5$, and $3$ vertices of degree $4$.
Going through the cases:
In Case 1 and Case 2  we have  $6$ vertices of configuration $(\alpha \alpha \alpha \alpha \alpha)$ so the corner $A$ appears $6\times 5 = 30$ times altogether, an impossibility.
In Case 3 we have 6 vertices of type $(\alpha \beta \alpha \beta \alpha)$ and 3 of type $(\alpha \beta \alpha \beta)$,
so the corner $A$ appears $6$ times, while the corners $B$ and $C$ appear $16$ times each.
In Case 4 we have $12$ instances of $B$ and $C$ and $18$ instances of $A$.
In Case 5 we have 6 vertices of configuration $(\alpha \beta \alpha \alpha \alpha)$ and the other 3 vertices
can be either of type $(\alpha \beta \alpha \beta)$ or $(\alpha \alpha \alpha \alpha)$.  And this gives at least $18$
vertices of type $A$.   
\end{proof}


\section{Monotileable surfaces with at most three edges}
\subsection{Prototiles with $1$ or $2$ edges}
A prototile with only one edge can be constructed by deforming the boundary of a hemisphere so that it fits together in only one way (Figure \ref{fig:twoedge}).   Any other tile with only one edge also has the same property that it fits together to form a topological sphere.
Note that a hemisphere, or any topological disc with boundary a circle and normal perpendicular to the plane of the curve along the boundary, is not a finite edge tile, because there are uncountably many ways that it can surround itself.

\begin{figure}[h!tbp]
	\begin{center}
		\includegraphics[height=20mm]{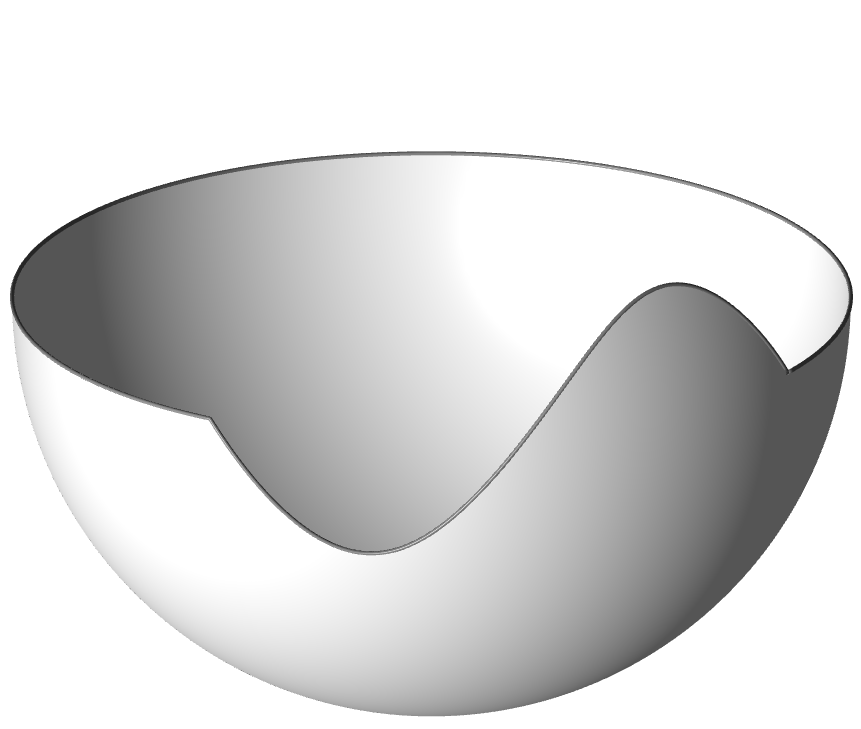}  \,
		\includegraphics[height=22mm]{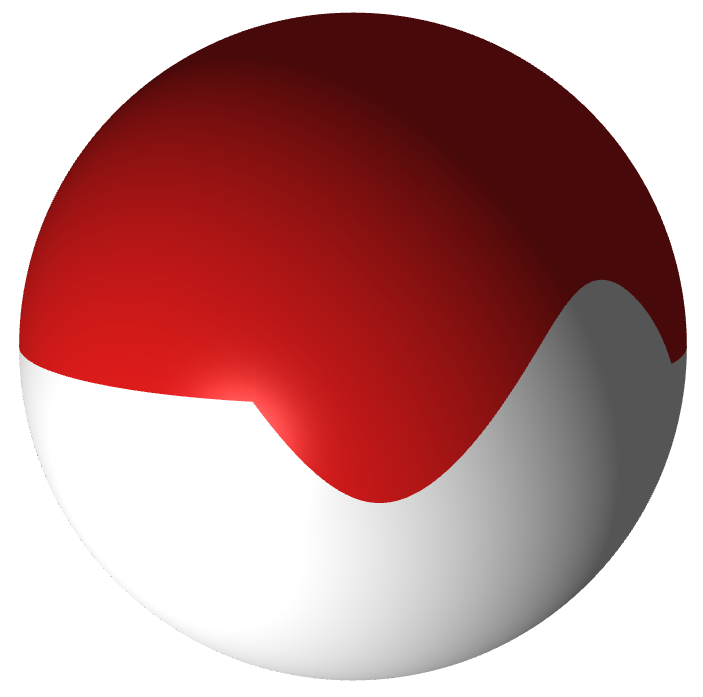}  \quad \quad \quad\quad
		\includegraphics[height=22mm]{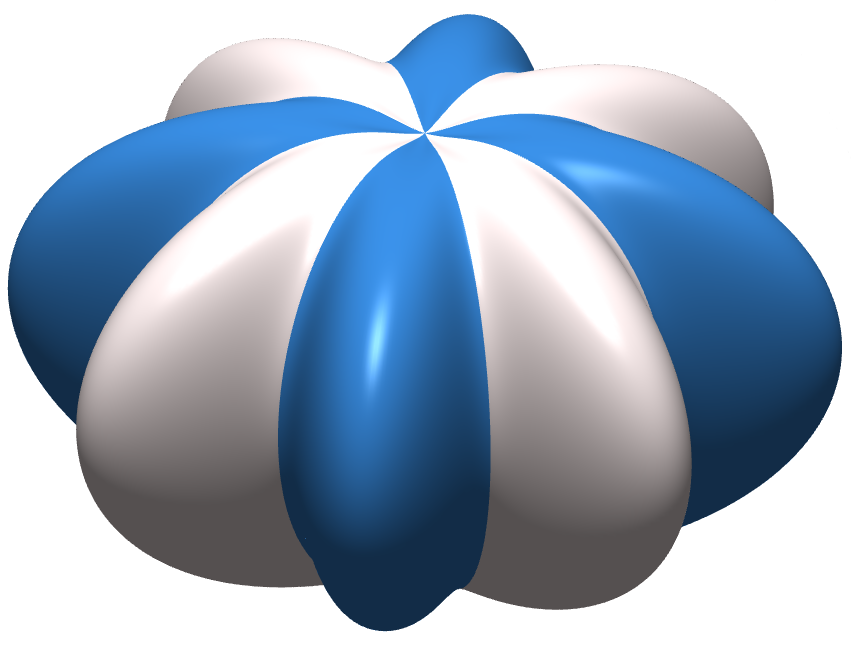} 
		\end{center}
	\caption{Left: 1 edge tile.  Right: 2 edges. }
	\label{fig:twoedge}
\end{figure}
The only configurations possible for a tile with two edges are $a \bar a$ and $\alpha \alpha$.
For $a \bar a$, if we denote the two vertices by $p$ and $q$,  and noting that $\bar a$, which runs
from $p$ to $q$ is reverse congruent to $a$, which runs from $q$ to $p$, it follows that the congruence is given 
by a rotation about the axis through $p$ and $q$.  Since the interior of the tile is by definition non-empty, this rotation is by some nonzero angle $\theta$.  There is only one way for the tiling to continue, and that is to repeatedly
rotate the prototile by the same angle $\theta$ about this axis.  If we require the surface to be embedded,
the angle must be of the form $\theta = 2\pi/m$, where $m$ is an integer, and the surface is a tiling of a topological
sphere by $m$ copies of this tile. We also have a restriction that the tangent to the edge curve $a$ at the end points must be perpendicular to the axis, and that the tangent to the tile at the vertices is the plane perpendicular to the axis.  The set of such surfaces (if embedded) is the set of compact surfaces of genus zero that are invariant under a  rotation of finite order (Figure \ref{fig:twoedge}). For $\alpha  \alpha$, one also has the possibility to turn the tile around, reversing $p$ and $q$, which gives various shapes depending on the interior of the tile, but the result is still a topological sphere.

\subsection{Prototiles with 3 edges} \label{sec:3edge}
As mentioned above, there are four different types of interior tile with three edges, depicted in
Figure \ref{fig:triangle4cases}.   The symmetries of the tiles and the maps that take a tile to an adjacent 
tile can be used to study these tilings.

The basis of our analysis is the following, which applies to the tiles of Type 1, 2 and 3 in 
Figure  \ref{fig:triangle4cases}:
\begin{lemma}
  \label{thm:axis_in_bisector}
  Suppose we are given two points $B\neq C$ in 3-space and two axes of
  rotation $\ell_1$ and $\ell_2$ with corresponding rotations $R_1$
  and $R_2$ such that $R_1(\phi)B=R_2(\pi)B=C$. Then both $\ell_1$
  and $\ell_2$ lie in the bisector of the points $B$ and $C$, and
  $\ell_2$ goes through the point $(B+C)/2$.
\end{lemma}
\begin{proof}
The bisector of $BC$ is the plane consisting of points equidistant from $B$ and $C$. The orbits of a rotation consist of points equidistant from its axis. Hence the axes lie in the bisector of $B$ and $C$.
  The rotation $R_2(\pi)$ is a reflection in the line $\ell_2$ and if
  $B$ is mapped to $C$ then $\ell_2$ intersects the line segment $BC$
  orthogonally in the midpoint. 
\end{proof}
We immediately have:
\begin{corollary} \label{casescor}
  With the same assumptions as in Lemma~\ref{thm:axis_in_bisector}, it follows that
  $\ell_1$ and $\ell_2$ are either equal, parallel or intersect at a
  point.
\end{corollary}
\begin{lemma}
  \label{thm:two_rotations}
 Suppose we have the same situation as in
  Lemma~\ref{thm:axis_in_bisector}, and suppose further that 
  there is a line
  $\ell_B$ through $B$ that is mapped by both rotations to a line $\ell_C$ through $C$.
  Then we have one of the
  following three possibilities:
  \begin{itemize}
  \item  $\ell_1=\ell_2$ and $\phi=\pi$, i.e.,
    $R_1(\phi)=R_2(\pi)$. The lines $\ell_B$ and $\ell_C$ are
    symmetric around $\ell_1=\ell_2$.
  \item The lines $\ell_1$, $\ell_2$, $\ell_B$, and $\ell_C$ are
    parallel.
  \item The lines $\ell_1$, $\ell_2$, $\ell_B$, and $\ell_C$ intersect
    in a single point $O$.
  \end{itemize}
\end{lemma}
\begin{proof}
  In all cases $R=R_2(\pi)^{-1}R_1(\phi)=R_2(\pi)R_1(\phi)$ is a
  rotation preserving $B$ and $\ell_B$. So it is a rotation around $\ell_B$.

  Consider the three possibilities of Corollary \ref{casescor}. 
  If $\ell_1=\ell_2$ then $R$ is a rotation through the angle
  $\pi+\phi$ around the axis $\ell_1=\ell_2$. This can only be the case
  if either $\ell_B=\ell_1=\ell_2$ or if $\phi=\pi$. That is,
  either $B=C$ and $\ell_B=\ell_C=\ell_1=\ell_2$ or $\phi=\pi$.

  If $\ell_1\neq\ell_2$ are parallel then $R$ is a non-trivial rotation 
  that preserves the plane orthogonal to $\ell_1$ and
  $\ell_2$. So $\ell_B$, the axis of rotation, must be parallel to
  $\ell_1$ and $\ell_2$ and then the same is true for $\ell_C=R_2(\pi)\ell_B$.

  Finally, if $\ell_1$ and $\ell_2$ intersect at a point $O$ then $R$
  is a non-trivial rotation that preserves $O$. As
  $R_2(\pi)$ preserves $O$ but not $B$ we have $B\neq O$ and $\ell_B$
  is a line through $O$ and $B$ and then $\ell_C=R_2(\pi)\ell_B$ is a
  line through $O$ and $C$.
\end{proof}

\begin{figure}[htp]
  \centering
  \unitlength=.2\textwidth
  \begin{tabular}{cccc}
    \includegraphics[height=\unitlength]{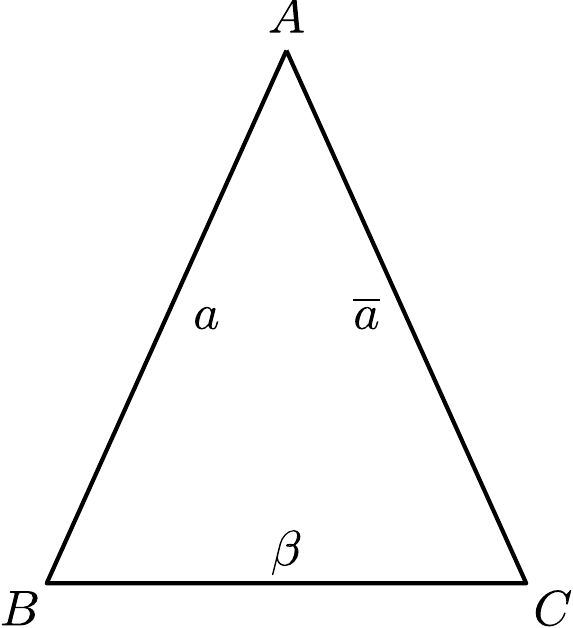}
    &\includegraphics[height=\unitlength]{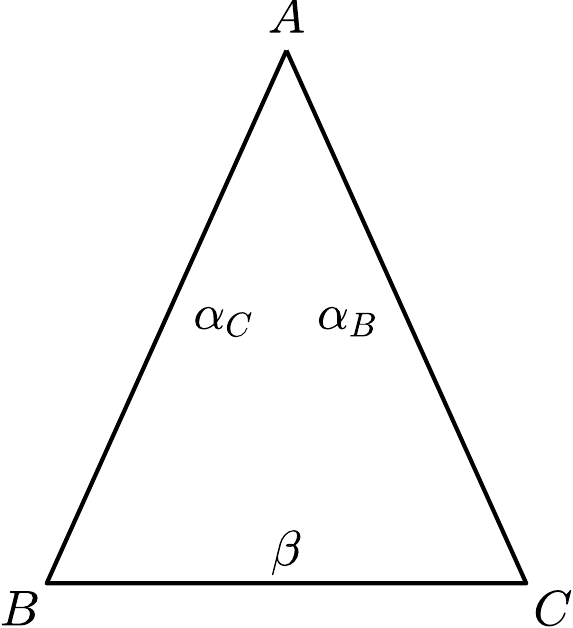}
    &\includegraphics[height=\unitlength]
      {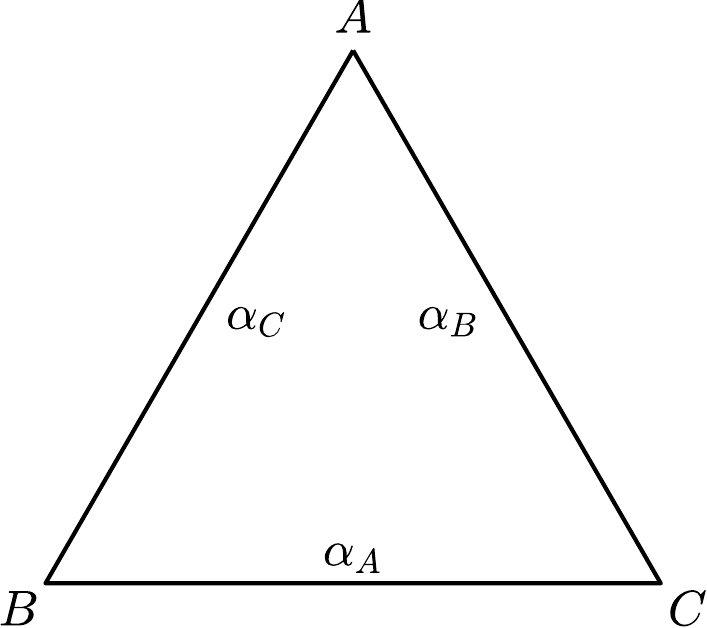}
    &\includegraphics[height=\unitlength]
      {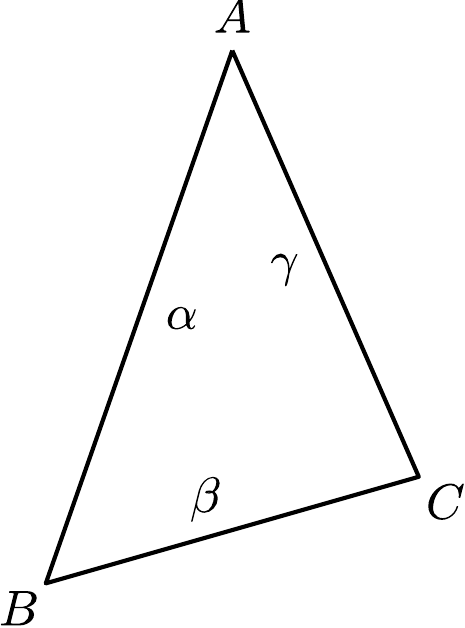}
    \\
    Type 1&Type 2&Type 3&Type 4
  \end{tabular}

  \caption{The four interior cases.}
  \label{fig:triangle4cases}
\end{figure}

We now suppose we have a monotiling of a complete surface by a 3-edge tile. 
In all four cases (Figure~\ref{fig:triangle4cases}) we have a normal line at each corner: $\ell_A$,
$\ell_B$, and $\ell_C$. We furthermore have a symmetry axis:
$\ell_{\alpha}$, $\ell_{\beta}$, and $\ell_{\gamma}$ at the symmetric
edges ($\alpha$, $\alpha_A$, $\alpha_B$, $\alpha_C$, $\beta$, and
$\gamma$). 
This is the situation of Lemma ~\ref{thm:two_rotations}. For the first case of the lemma
 the rotation $R_A(\phi)=R_{\beta}(\pi)$,
  where $\phi=\pi$, maps the boundary of the triangle to itself. The result is a surface with just two
  copies of the same tile, like a $1$-edge tile. 
  We thus exclude this uninteresting case from the definition of Types 1-4.  
  Lemma ~\ref{thm:two_rotations} then implies:
\begin{lemma}
  \label{lemma:normals}
  For tiles of type 1, 2, and 3 either:
  \begin{enumerate}
  \item the lines $\ell_A$, $\ell_B$, $\ell_C$, and $\ell_{\beta}$
    are parallel, or
  \item the lines $\ell_A$, $\ell_B$, $\ell_C$, and $\ell_{\beta}$
    ($\ell_{\alpha_A}$ in Type 3)
    intersect at a common point $O$ lying in the bisector of the edge
    $BC$.
  \end{enumerate}
\end{lemma}

This allows a description of all possible tilings of Type 1:
\begin{theorem}  \label{thm:type1}
  Let $S$ be complete monotiled surface tiled by a 3-edge tile of
  Type~1. Then the polyhedral surface formed from the vertices of $S$ is either:
  \begin{enumerate}
    \item A surface obtained from one of the three regular tilings of the plane (equilateral triangle, square,  regular hexagon),
    by lifting the center of each face the same height above the plane to form a pyramid.
  \item Either a symmetric bipyramid over a regular polygon, or a surface obtained from one of the five platonic solids (tetrahedron, cube,
    octahedron, dodecahedron, or icosahedron),  by attaching/excavating a pyramid to/from each face. The 
    apexes of the pyramids all lie on a sphere concentric with the platonic solid.
  \end{enumerate}
\end{theorem}
\begin{proof} 
 Consider the two cases of Lemma~\ref{lemma:normals}:
  If the lines are parallel, a Euclidean motion mapping one tile to another
  is a rotation around an axis parallel to $\ell_A$ and $\ell_{\beta}$
  hence the orbit of $A$ is contained in a plane and the orbits of $B$
  and $C$ are equal and contained in the same or a parallel plane. If
  we look at the vertex $A$ then the neighboring vertices lies in the
  orbit of $B$ and $C$, in fact in the orbit generated by
  $R_{A}(\phi)$. The edges are images of $\beta$ and hence the vertices
  form a regular polygon in the plane with $A$ as its center. 
  The orbits of $\beta$ thus form a regular tiling of the plane.
  
  If the lines meet at a point $O$, any Euclidean motion mapping one tile to another is
  a rotation around an axis through the point $O$. Now the orbit of
  $A$ is contained in a sphere with center $O$ and the orbits of $B$
  and $C$ are equal and contained in the same or a concentric
  sphere. Again the neighbouring vertices form a planar regular
  polygon but now it also lies on a sphere with center $O$.  Thus,
the orbits of $\beta$ form a regular tiling of the sphere and hence a Platonic solid, or a symmetric bipyramid 
over a single copy of the planar polygon.
\end{proof}
\begin{figure}[h!tbp]
	\begin{center}
		\includegraphics[height=25mm]{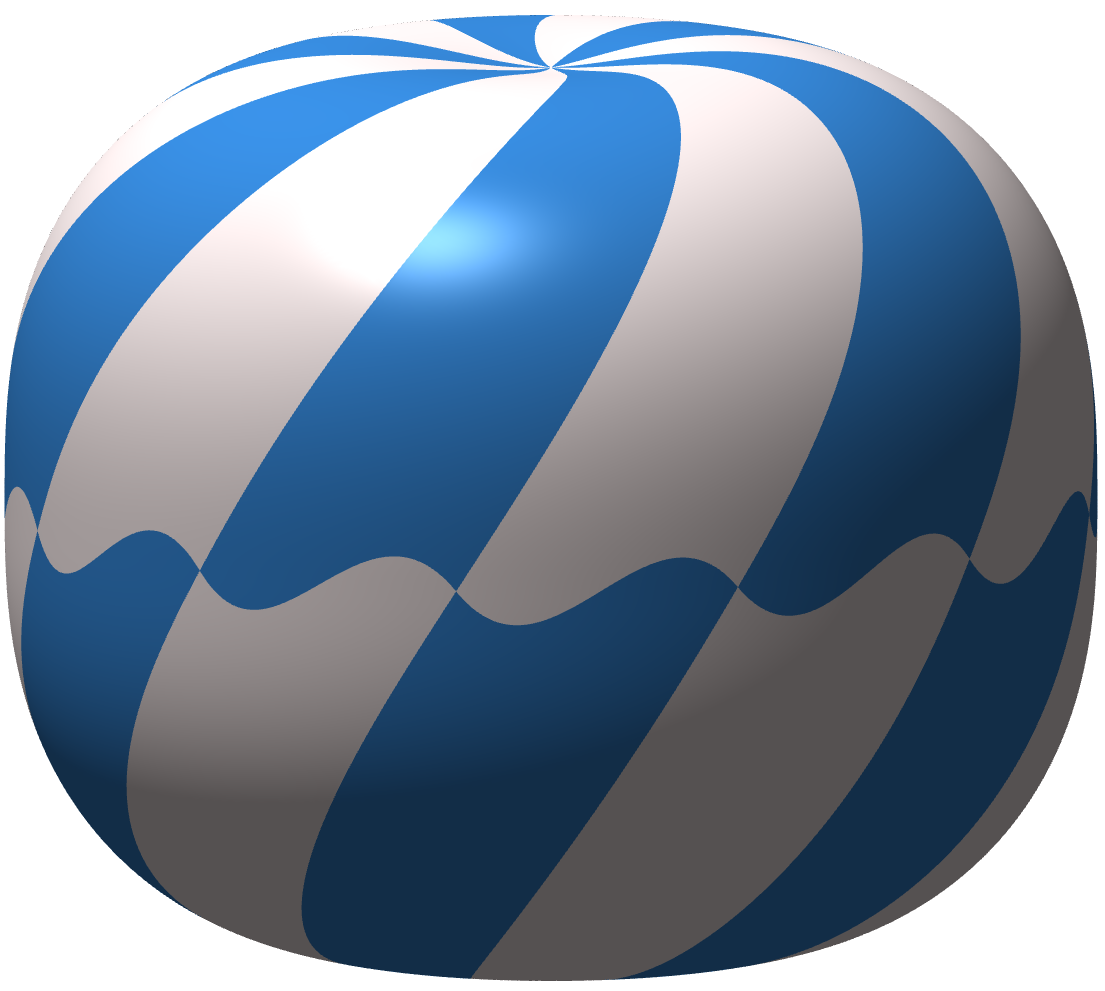} \quad \quad
		\includegraphics[height=25mm]{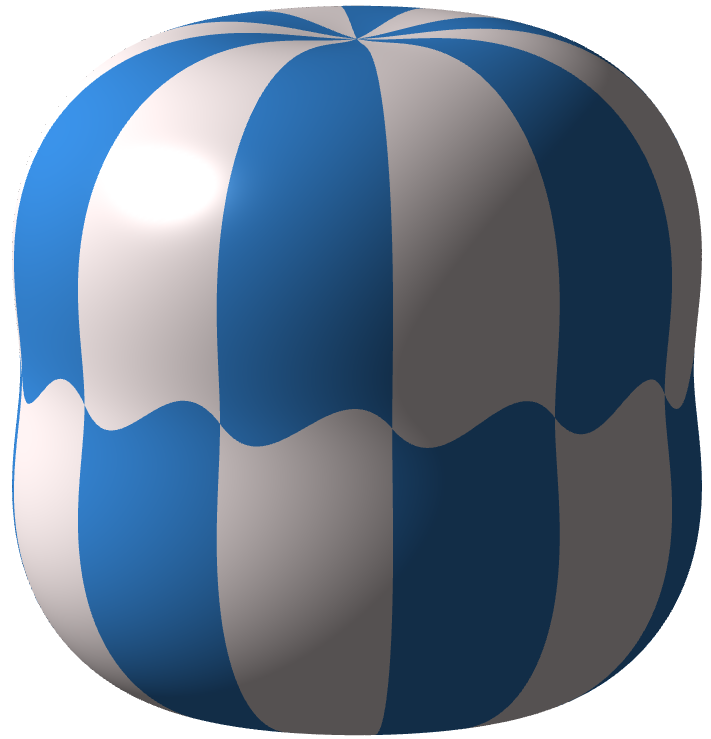} \quad \quad
		\includegraphics[height=26mm]{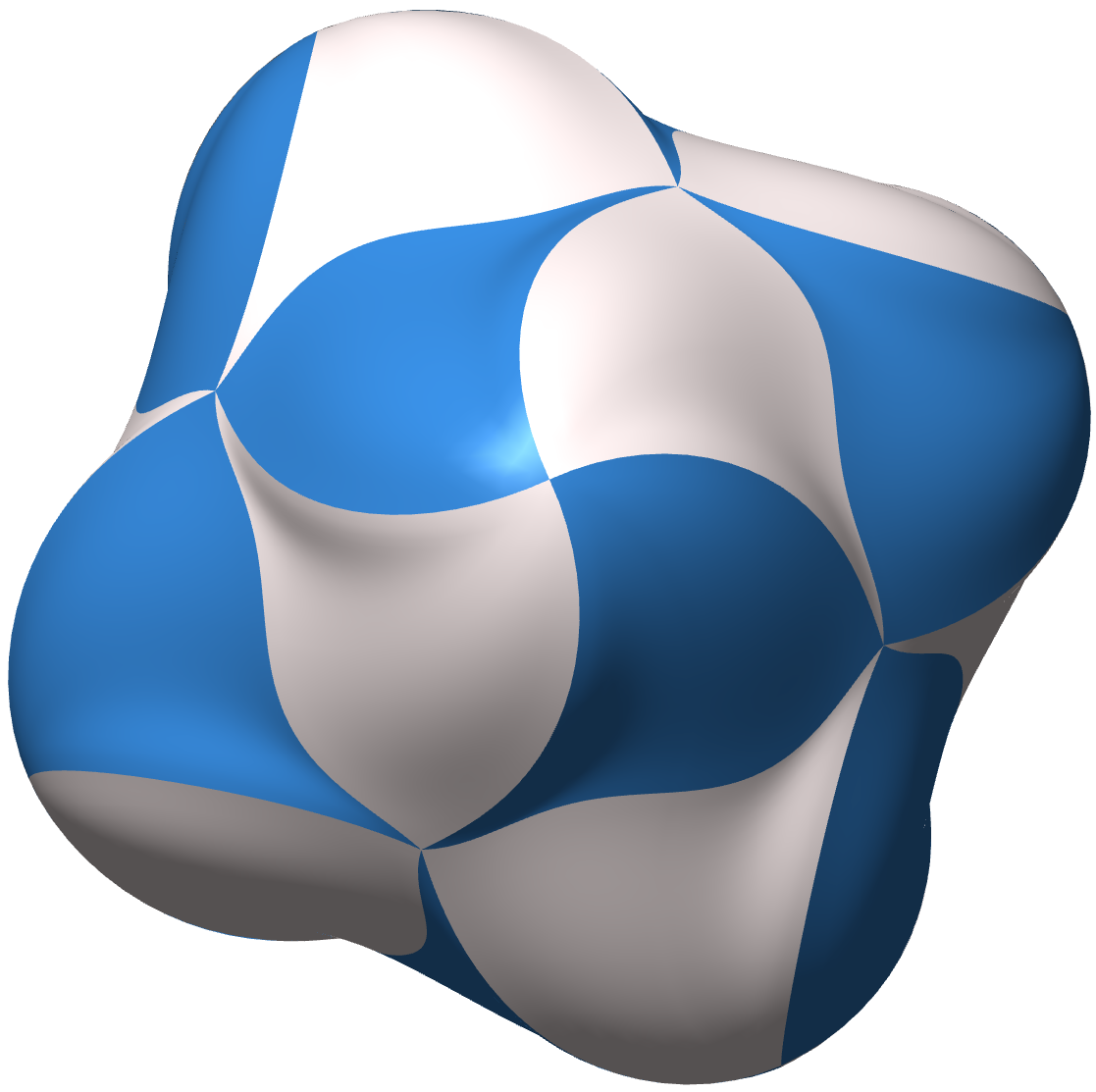} \quad \quad
		\includegraphics[height=25mm]{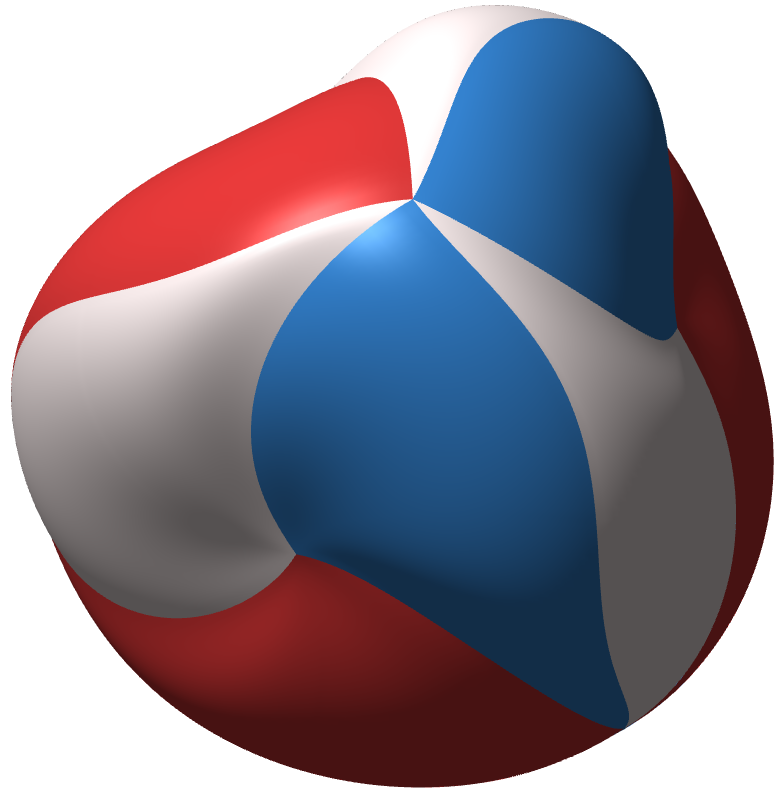}
		\end{center}
	\caption{Left to right: type $a \bar a \beta$, dodecagonal bipyramid structure;
	  type $\alpha \alpha \beta$, dodecagonal bipyramid;
	   type $a \bar a \beta$, tetrakis cube structure;
	   type $a \bar a \beta$ excavated tetrahedron.}
	\label{fig:threeedge}
\end{figure}

Note: surfaces of the type described in Theorem \ref{thm:type1} are easy to construct
by taking a regular tiling of the plane or sphere and then raising the center of each tile smoothly and
respecting the symmetry of the edges, see Figure \ref{fig:threeedge}, right.

\begin{lemma}
  \label{lemma:normals2}
  For tiles of type 2 and 3 we have:
  \begin{enumerate}
  \item If $\ell_A$, $\ell_B$, $\ell_C$, and $\ell_{\beta}/\ell_{\alpha_A}$
    are parallel, then the lines
   $\ell_{\alpha_B}$, and $\ell_{\alpha_C}$ are either both parallel to or both
   orthogonal to the other four lines.  In the case of type 3,  all six lines are parallel.
  \item If the lines $\ell_A$, $\ell_B$, $\ell_C$, and $\ell_{\beta}$
    intersect at a common point $O$, then the lines $\ell_{\alpha_B}$ and $\ell_{\alpha_C}$
    also intersect at the point $O$.
  \end{enumerate}
\end{lemma}
\begin{proof}
For tiles of type 2,  the reflection around $\ell_{\alpha_B}$ maps the line
  $\ell_A$ to the parallel line $\ell_C$. Hence $\ell_{\alpha_B}$ has
  to be parallel with $\ell_A$ or orthogonal to $\ell_A$. As a
  rotation around $\ell_A$ maps $\ell_{\alpha_C}$ to $\ell_{\alpha_B}$
  they are either both parallel with $\ell_A$ or both orthogonal to
  $\ell_A$. In the first case the reflections around the parallel
  lines $\ell_{\alpha_B}$ and $\ell_{\alpha_C}$ maps $A$ to C and B,
  respectively, so the points $A$, $B$, and $C$ lie in an orthogonal
  plane.   For Type 3, $R_B(\phi)$ maps $\ell_{\alpha_A}$ to $\ell_{\alpha_C}$ and
  $R_C(\phi)$ maps $\ell_{\alpha_B}$ to $\ell_{\alpha_A}$.  Since $\ell_{\alpha_A}$ is parallel 
  to both axes of rotation,  so are  $\ell_{\alpha_B}$ and $\ell_{\alpha_C}$.

  If $\ell_A$ and $\ell_{\beta}$ intersect at a point $O$ then
  $R_A(\pm\phi)R_{\beta}(\pi)$ preserves $O$ and hence both $\ell_B$
  and $\ell_C$ goes through $O$.
The reflection around $\ell_{\alpha_B}$ maps the
  line $\ell_A$ to the line $\ell_C$ so it preserves the point $O$,
  i.e., the line $\ell_{\alpha_B}$ goes through $O$. Similar, the line
  $\ell_{\alpha_C}$ goes through $O$.
\end{proof}
This gives a restriction on possible monotiled surfaces for 3-edge tiles of Type 2:
\begin{theorem}  \label{thm:type2}
  Let $S$ be complete monotiled surface tiled by 3-edge tile of
  Type~2.  Then the polyhedral surface formed from the vertices of $S$ is either:
  \begin{enumerate}
  \item An arbitrary edge to edge monotiling of the plane by an isosceles triangle. 
  Furthermore, the normals at the vertices and midpoints of the edges are equal and
    orthogonal to the plane.
  \item \label{cylinder}
The orbit of $A$ lies in a plane and the orbits of $B$ and
    $C$ lies in the same or a parallel plane. The triangle formed by the vertices of 
    a tile is isosceles.  Furthermore, the
    normals at the vertices in the orbit of $A$ are equal and
    orthogonal to the plane and the normals at the vertices in the
    orbit of $B$ and $C$ are the opposite vector.
      \item The underlying polyhedron is an arbitrary triangulation of the
    sphere with copies of a single isosceles triangle. Furthermore,
    the normal at the vertices are orthogonal to the sphere.
  \end{enumerate}
  \end{theorem}
\begin{proof}
As in the proof of Theorem  \ref{thm:type1}, if the normals are parallel,
we find that the orbit of $A$ is in one plane, and the orbits of $B$ and $C$ are equal and contained in the
same or a parallel plane.
  If the triangular tile is of type~2 or 3 and all six lines $\ell_A$,
  $\ell_B$, $\ell_C$, $\ell_{\beta}$, $\ell_{\alpha_B}$, and
  $\ell_{\alpha_C}$ are parallel then the reflection around
  $\ell_{\alpha_B}$ maps $A$ to $C$ and hence the orbit containing $A$
  lies in the same plane as the orbit containing $B$ and $C$. That
  implies that the polyhedron is a triangulation of the plane with a
  single isosceles triangle. The normals at a vertex $X$ lie on the
  line $\ell_X$ and as the reflections above maps them to each other
  they are all the same and orthogonal to the plane. 

If  the lines $\ell_{\alpha_B}$
  and $\ell_{\alpha_C}$ are orthogonal to the four parallel lines
  $\ell_A$, $\ell_B$, $\ell_C$, and $\ell_{\beta}$, then the orbit of
  $A$ and the orbit of $B$ and $C$ may lie in two different but
  parallel planes. The normals at all vertices are parallel with the
  line $\ell_A$. The rotation around $\ell_{\alpha_B}$, and
  $\ell_{\alpha_C}$ maps the normals at the orbits of $A$ to the
  normals at the vertices at the orbit of $B$ and $C$. Hence, the
  normal at the orbit of $A$ is the opposite of the normal at the
  vertices of $B$ and $C$.

  If  all six lines $\ell_A$,
  $\ell_B$, $\ell_C$, $\ell_{\beta}$, $\ell_{\alpha_B}$, and
  $\ell_{\alpha_C}$ intersect at the point $O$ then the reflection
  around $\ell_{\alpha_B}$ maps $A$ to $C$ and hence the orbit
  containing $A$ lies on the same sphere as the orbit containing $B$
  and $C$. That implies that the polyhedron is a triangulation of the
  sphere with a single isosceles triangle. 
\end{proof}
Note that Item \ref{cylinder} of Theorem \ref{thm:type2} includes infinite round cylinders. It is not clear to the authors whether other complete solutions exist.

Similarly we obtain:
\begin{theorem}  \label{thm:type3}
  Let $S$ be complete monotiled surface tiled by 3-edge tile of
  Type~3.  The polyhedral surface formed from the vertices of $S$ is either:
  \begin{enumerate}
  \item An arbitrary edge to edge monotiling of the plane by an equilateral triangle. 
  Furthermore, the normals at the vertices and midpoints of the edges are equal and
    orthogonal to the plane.
      \item The underlying polyhedron is an arbitrary triangulation of the
    sphere with copies of a single equilateral triangle. Furthermore,
    the normal at the vertices are orthogonal to the sphere.
  \end{enumerate}
  \end{theorem}
Note that, if the interior of the tile is not symmetric, then the tiling possibilities described in Theorem \ref{thm:type3}
can be augmented by placing each tile in any of three positions.

The last case of $3$-edge monotilings is case $4$, and, by Theorem \ref{rigiditythm} on rigid tiles we 
immediately have:
\begin{theorem}  \label{thm:type4}
  Let $S$ be complete monotiled surface tiled by a 3-edge tile of
  Type~4.  Then  the group $G$ of rotations generated by the three
  rotations $R_\alpha$, 
  $R_\beta$ and $R_\gamma$, of angle $\pi$, around the symmetry axes of the three edges, acts transitively
  on the tiles of $S$. The composition $R=R_\beta R_\gamma R_\alpha$ is of finite order.
  \end{theorem}
Note  that a surface of Type $4$ has only one type of vertex, and this is of degree $3k$, where $k$ is the 
order of $R$.  Examples are a sphere (tiled by, for example, 4 triangles) where $R$ is of order $1$, and a plane or round cylinder, where $R$ is of order $2$.

\section{Conclusions and future directions}
We have shown that tileable surfaces are a genuinely distinct phenomenon from both tilings of constant curvature surfaces and monohedral (or $k$-hedral) polyhedral surfaces.   An important difference from tilings of space forms is that the choice of prototile (or prototile set) changes the \emph{shape} of the surface. This should have implications in terms of practical relevance.

In comparison with polyhedral surfaces, smooth tileable surfaces have the attractive feature that they look smooth, but they offer the challenge that they are more difficult to create, due to matching the normals along the edges.  We have so far left unexplored finite edge tilings with more than three edges, and tilings that are not finite edge type, as well as larger prototile sets. We expect that more interesting surfaces are obtainable when these are investigated.

There are many interesting problems for future research, and we list just a few here:
\begin{enumerate}
\item Given an admissible tile, what type of surface(s) does it produce (compact, complete, or incomplete). 
This is related to the curved  surface version  of the \emph{Heesch problem} \cite{heesch1968regulares} \cite{mann2004heesch} for plane tilings: given an admissible tile it can be surrounded at least once by a layer of tiles 
(a corona).  If the resulting collection cannot be completely surrounded, then the Heesch number is $1$.  In general, the Heesch number is the number of coronas that can be added.  Heesch's tiling problem is: given a positive integer $N$, does there exist a tile with Heesch number $N$? For tileable surfaces we could consider surfaces with or without self-intersections, which would modify the answer for these questions.
\item
It is known {\cite{smith2023aperiodic}} that for tilings of the plane one can find a monotile that generates non-repeating patterns.  
For curved tileable surfaces, it is easy to give examples of monotiles that generate non-repeating patterns if 
we consider tiles that are not of finite edge type.  Finding finite edge type non-repeating tiles that are not just deformed  versions of the planar versions may be an interesting problem.
\item 
Is there a simple criterion that tells you when a monohedral polyhedron can be smoothed into a tiled surface?
\item Finding useful algorithms for producing tileable surfaces with various global surface shapes would be very interesting for applications.  Especially, if we increase the number of prototiles, one can expect to generate a large variety of shapes. 
For applications in architecture, one is typically interested in a surface with boundary.  Various boundary problems related to tileable surfaces are of interest.
\end{enumerate}



\bibliographystyle{amsplain}
\bibliography{mybib}

\end{document}